%% file: Dilation_to_tetrablock_unitaries.tex
\documentclass[12pt, twoside, a4paper]{amsart}
\usepackage{amsmath,amsthm,amsopn,amssymb,a4wide,varioref,amscd}

\usepackage{bm}
\usepackage{hyperref}
\newcommand{\comment}[1]{}

\newcommand{\clh}{\mathcal{H}}
\newcommand{\clk}{\mathcal{K}}

\input{mypreamble}

\numberwithin{equation}{section}

\begin{document}

\title[]{Explicit and unique construction of tetrablock unitary dilation in a certain case}
%
%\author{ Tirthankar Bhattacharyya}
%
%\address{Department of Mathematics,\\
%        Indian Institute of Science,\\
%        Bangalore 560012, India}
%
%\email{tirtha@math.iisc.ernet.in}
\author[Bhattacharyya]{ T. Bhattacharyya}

\address{Department of Mathematics,\\
        Indian Institute of Science,\\
        Bangalore 560012, India}

\email{tirtha@math.iisc.ernet.in}
\author[Sau]{H. Sau}
\address{Department of Mathematics,\\
        Indian Institute of Science,\\
        Bangalore 560012, India}

\email{sau10@math.iisc.ernet.in}
\thanks{The authors' research is supported by Department of Science and
Technology, India through the project numbered SR/S4/MS:766/12 and
University Grants Commission, India via DSA-SAP.}
\date{\today}
\maketitle

\begin{abstract}
 Consider the domain $E$ in $\mathbb{C}^3$ defined by
 $$
 E=\{(a_{11},a_{22},\text{det}A): A=\begin{pmatrix} a_{11} & a_{12} \\ a_{21} & a_{22} \end{pmatrix}\text{ with }\lVert A \rVert <1\}.
 $$
 This is called the tetrablock. This paper constructs explicit boundary normal dilation for a triple $(A,B,P)$ of commuting bounded operators which has $\bar{E}$ as a spectral set.

 We show that the dilation is minimal and unique under a certain natural condition. As is well-known, uniqueness of minimal dilation usually does not hold good in several variables, e.g., Ando's dilation is not known to be unique. However, in the case of the tetrablock, the third component of the dilation can be chosen in such a way as to ensure uniqueness.
\end{abstract}

\section{Introduction}
Let $K$ be a compact subset of $\mathbb{C}^d$ for $d \geq 1$. Consider a $d$-tuple $\underline{T}=(T_1,T_2,\dots T_d)$ of commuting bounded operators with $K$ as a spectral set, i.e., the joint spectrum of $\underline{T}$ is contained in $K$ and
$$ ||f(\underline{T})|| \leq \text{sup}\{ |f(z)|: z \in K \} $$
for all rational functions $f$ with poles off $K$. A commuting tuple of bounded normal operators $\underline{N}=(N_1,N_2,\dots N_d)$ with $\sigma$({${\underline{N}}$})$\subset bK$, the distinguished boundary of $K$ is called {\em{a normal boundary dilation}} of $\underline{T}$ if
$$ f(\underline{T}) = P_\mathcal{H} f(\underline{N})|_\mathcal{H},$$
for all rational functions $f$ with poles off $K$. The $K$ we consider in this paper, is a polynomially convex domain. Note that by Oka-Weil theorem, for a polynomially convex domain, a polynomial dilation is the same as a rational dilation. In other words,
$$ T_1^{k_1}\cdots T_d^{k_d}=P_\mathcal{H}N_1^{k_1}\cdots N_d^{k_d}|_\mathcal{H}$$
for $k_1,\dots, k_d \geq 0$.

\begin{definition}
 A triple $(A,B,P)$ of commuting bounded operators on a Hilbert space $\mathcal{H}$ is called a {\em{tetrablock contraction}} if $\overline{E}$ is a spectral set for $(A,B,P)$, i.e., the joint spectrum of $(A,B,P)$ is contained in $\overline{E}$ and
$$
||f(A,B,P)|| \leq ||f||_{\infty,\overline{E}}=\text{sup}\{ |f(x_1,x_2,x_3)|:(x_1,x_2,x_3) \in \overline{E}\}
$$
for any polynomial $f$ in three variables.
\end{definition}

Consider a tetrablock contraction $(A,B,P)$. Then it is easy to see that $P$ is a contraction.
Fundamental equations for a tetrablock contraction are introduced in \cite{sir's tetrablock paper}. These are
\begin{eqnarray}\label{Maa12}
A-B^*P=D_PF_1D_P, \text{ and }  B-A^*P=D_PF_2D_P
\end{eqnarray}
where $D_P=(I-P^*P)^\frac{1}{2}$ is the defect operator of the contraction $P$ and $\mathcal{D}_P=\overline{Ran}D_P$ and $F_1,F_2$ are bounded operators on $\mathcal{D}_P$.
Theorem 3.5 in \cite{sir's tetrablock paper} says that the two fundamental equations can be solved and the solutions $F_1$ and $F_2$ are unique. The unique solutions
$F_1$ and $F_2$ of equations (\ref{Maa12}) are called the {\em{fundamental operators}} of the tetrablock contraction $(A,B,P)$. Moreover, $w(F_1)$ and $w(F_2)$ are not greater than $1$.
The adjoint triple $(A^*,B^*,P^*)$ is also a tetrablock contraction as can be seen from definition. By what we stated above there are unique $G_1,G_2 \in \mathcal{B}(\mathcal{D}_{P^*})$ such that
\begin{eqnarray}\label{Maa13}
A^*-BP^*=D_{P^*}G_1D_{P^*} \text{ and } B^*-AP^*=D_{P^*}G_2D_{P^*},
\end{eqnarray}
and $w(G_1)$ and $w(G_2)$ are not greater than $1$.

\begin{definition}
A {\em{tetrablock unitary}} is a triple of commuting bounded operators $\underline{N}=(N_1,N_2,N_3)$ on a Hilbert space $\mathcal{H}$ such that its Taylor joint spectrum $\sigma({\underline{N}})$ is contained in $bE$, the Shilov boundary of $E$.
 \end{definition}
 \begin{definition}
 A {\em{tetrablock isometry}} is the restriction of a tetrablock unitary to a joint invariant subspace.
 \end{definition}

Let $(A,B,P)$ be a tetrablock contraction on $\clh$ with fundamental operators $F_1$
and $F_2$. Consider the Hilbert space $\tilde{\mathcal{H}}=\mathcal{H} \oplus l^2(\mathcal{D}_P)$. Let $V_1$, $V_2$ and $V_3$ be defined on
$\tilde{\mathcal{H}}$ by
$$ V_1(h \oplus (a_0,a_1,a_2,\dots))=(Ah \oplus (F_2^*D_Ph+F_1a_0, F^*_2 a_0 + F_1a_1, F^*_2 a_1 + F_1a_2, \dots ))
$$
$$V_2(h \oplus (a_0,a_1,a_2,\dots))=(Bh \oplus (F^*_1D_Ph + F_2a_0, F_1^* a_0 + F_2a_1, F_1^* a_1 + F_2a_2, \dots ))
$$
$and$
$$
V_3(h \oplus (a_0,a_1,a_2,\dots))=(Ph \oplus (D_Ph,a_0,a_1,a_2,\dots))
$$
respectively. From Theorem 6.1 of \cite{sir's tetrablock paper}, we learnt that $(V_1,V_2,V_3)$ on $\tilde{\mathcal{H}}$ is a tetrablock isometric dilation of $(A,B,P)$ if $F_1,F_2$ satisfy
\begin{eqnarray}\label{Maa14}
[F_1, F_2] =0 \text{ and } [F_1, F_1^*] = [F_2, F_2^* ].
\end{eqnarray}
Our first major result, described in the theorem below,  is the construction of tetrablock unitary dilation of a tetrablock contraction explicitly.
\\
\begin{theorem}\label{1st major thm}
Let $(A,B,P)$ be a tetrablock contraction on $\clh$ with fundamental operators $F_1$
and $F_2$ satisfying (\ref{Maa14}). Let the space $\tilde{\mathcal{H}}$ and the operator triple $(V_1,V_2,V_3)$ be as above. Consider the space $\mathcal{K}= \tilde{\mathcal{H}} \oplus l^2({\mathcal{D}_{P^*}})$. Define operators $C_1,C_2, C_3:l^2({\mathcal{D}_{P^*}}) \to \tilde{\mathcal{H}}$ by
\begin{eqnarray*}
&&C_1(a_0,a_1,a_2,\dots)=(D_{P^*}G_2a_0 \oplus (-F_2^*P^*a_0,0,0,\dots)),\\
&&C_2(a_0,a_1,a_2,\dots)=(D_{P^*}G_1a_0 \oplus (-F_1^*P^*a_0,0,0,\dots)),\\
&&C_3(a_0,a_1,a_2,\dots)=(D_{P^*}a_0 \oplus (-P^*a_0,0,0,\dots))
\end{eqnarray*}
and $D_1,D_2,D_3:l^2({\mathcal{D}_{P^*}}) \to l^2({\mathcal{D}_{P^*}})$ by
\begin{eqnarray*}
&&D_1(a_0,a_1,a_2,\dots)=(G_1^*a_0+G_2a_1,G_1^*a_1+G_2a_2,G_1^*a_2+G_2a_3,\dots),\\
&&D_2(a_0,a_1,a_2,\dots)=(G_2^*a_0+G_1a_1,G_2^*a_1+G_1a_2,G_2^*a_2+G_1a_3,\dots)\text{ and}\\
&&D_3(a_0,a_1,a_2,\dots)=(a_1,a_2,a_3,\dots) \text{ respectively}.
\end{eqnarray*} Finally, let $R_1$, $R_2$ and $U$ be three operators defined on $\clk$ whose block operator matrices with respect to the decomposition $\mathcal{K}= \tilde{\mathcal{H}} \oplus l^2({\mathcal{D}_{P^*}})$ are
$$ \left(
     \begin{array}{cc}
       V_1 & C_1 \\
       0 & D_1 \\
     \end{array}
   \right), \left(
              \begin{array}{cc}
                V_2 & C_2 \\
                0 & D_2 \\
              \end{array}
            \right) \mbox{ and } \left(
                                   \begin{array}{cc}
                                     V_3 & C_3 \\
                                     0 & D_3 \\
                                   \end{array}
                                 \right) \mbox{ respectively}.
 $$
%\bordermatrix{~ & \tilde{\mathcal{H}} & l^2({\mathcal{D}_{P^*}}) \cr
 %                 \tilde{\mathcal{H}} & V_1 & C_1 \cr
  %                l^2({\mathcal{D}_{P^*}}) & 0 & D_1 \cr}
%,
%\bordermatrix{~ & \tilde{\mathcal{H}} & l^2({\mathcal{D}_{P^*}}) \cr
 %                 \tilde{\mathcal{H}} & V_2 &  C_2\cr
  %                l^2({\mathcal{D}_{P^*}}) & 0 & D_2 \cr}
%\mbox{ and }
%\bordermatrix{~ & \tilde{\mathcal{H}} & l^2({\mathcal{D}_{P^*}}) \cr
 %                 \tilde{\mathcal{H}} & V_3 & C_3 \cr
  %                l^2({\mathcal{D}_{P^*}}) & 0 & D_3 \cr}.
%$$

Then the triple $(R_1, R_2, U)$ is a tetrablock unitary dilation of $(A,B,P)$.
\end{theorem}

This dilation is proved to be minimal. Unlike in the case of a single operator, minimality of a dilation of an $n$-tuple$(n >1)$ of commuting operators does not guarantee its uniqueness. We show that the tetrablock unitary dilation $(R_1,R_2,U)$ defined in Theorem \ref{1st major thm} of a tetrablock contraction $(A,B,P)$ is unique under a certain suitable condition. The uniqueness theorem, the second major result of this paper, states that if $(\tilde{R}_1, \tilde{R}_2, \tilde{U})$ is a tetrablock unitary dilation of a tetrablock contraction $(A,B,P)$ such that $\tilde{U}$ is the minimal unitary dilation of $P$, then the dilation $(\tilde{R}_1, \tilde{R}_2, \tilde{U})$ is unitarily equivalent to the dilation we have constructed. This is the content of Theorem \ref{uniqueness}. Note that if a tetrablock unitary $(\tilde{R}_1, \tilde{R}_2, \tilde{U})$ dilates a tetrablock contraction $(A,B,P)$, then by definition of dilation, the unitary $\tilde{U}$ dilates the contraction $P$ too. So the only constraint we are imposing is that the dilation $\tilde{U}$ is minimal.

Two equations associated with a contraction $P$ and its defect operators that will come handy are
\begin{eqnarray}\label{Maa8}
PD_P=D_{P^*}P
\end{eqnarray}
and its corresponding adjoint relation
\begin{eqnarray}\label{Maa10}
D_PP^*=P^*D_{P^*}.
\end{eqnarray}
Proof of (\ref{Maa8}) and (\ref{Maa10}) can be found in \cite{Nagy-Foias}(ch. 1, sec. 3). We shall use these two relations in this paper without mention.

Note that for a Hilbert space $\mathcal{E}$, the Hilbert space $l^2(\mathcal{E})$ is unitarily equivalent to the Hilbert space $H^2_{\mathcal{E}}(\mathbb{D})$, via the unitary map
$$(\xi_0,\xi_1, \xi_2 \dots) \mapsto \sum_{n=0}^{\infty}z^n\xi_n,$$
 where $\xi_n \in \mathcal{E}$ for all $n \geq 0$.

\section{Elementary Results On A Tetrablock Contraction}
All known facts about tetrablock contractions, tetrablock isometries and tetrablock unitaries that we quote here are from \cite{sir's tetrablock paper}.
%Recall that the set tetrablock is defined as
%$$
%E=\{\underline{x}=(x_1,x_2,x_3)\in \mathbb{C}^3: 1-x_1z-x_2w+x_3zw \neq 0 \text{ whenever }|z| \leq 1\text{ and }|w| \leq 1\}.
%$$
%\begin{definition}
%Let $(A,B,P)$ be a triple of commuting bounded operators on a Hilbert space $\mathcal{H}$. We call it a {\em{tetrablock contraction}} if $\overline{E}$ is a spectral set for $(A,B,P)$, i.e., the joint spectrum of $(A,B,P)$ is contained in $\overline{E}$ and
%$$
%||f(A,B,P)|| \leq ||f||_{\infty,\overline{E}}=\text{sup}\{ |f(x_1,x_2,x_3)|:(x_1,x_2,x_3) \in \overline{E}\}
%$$
%for any polynomial $f$ in three variables.
%\end{definition}

There are well known characterizations of a tetrablock isometry and a tetrablock unitary. We just quote some characterizations of a tetrablock unitary because we shall use it later in this paper.
\begin{theorem}\label{tetrathm1}
Let $\underline{N} = (N_1,N_2,N_3)$ be a commuting triple of bounded operators. Then
the following are equivalent:
\begin{enumerate}
\item[(1)] $\underline{N}$ is a tetrablock unitary;
\item[(2)] $N_3$ is a unitary, $N_2$ is a contraction and $N_1 = N_2^*N_3$;
\item[(3)] there is a $2 \times 2$ unitary block operator matrix $[U_{ij} ]$ where $U_{ij}$ are commuting
normal operators and $\underline{N} = (U_{11}, U_{22}, U_{11}U_{22} - U_{21}U_{12})$;
\item[(4)] $N_3$ is a unitary and $\underline{N}$ is a tetrablock contraction;
\item[(5)] the family $\{(R_z, U_z) : |z| = 1\}$ where $R_z = N_1+zN_2$ and $U_z = zN_3$ is a commuting
family of $\Gamma$-unitaries.
\end{enumerate}
\end{theorem}
For a proof, see \cite{sir's tetrablock paper}(Th. 5.4).
Before going to construct the tetrablock unitary dilation of a tetrablock contraction, let us study few lemmas which will be used in the construction. First we state a very important result from \cite{sir's tetrablock paper}(Corollary 4.2).
\begin{lemma}\label{tetra}
The fundamental operators $F_1$ and $F_2$ of a tetrablock contraction $(A,B,P)$
are the unique bounded linear operators on $\mathcal{D}_P$ that satisfy the pair
of operator equations
\begin{eqnarray*}
D_PA = F_1D_P + F_2^*D_PP \text{ and } D_P B = F_2D_P + F_1^*D_PP.
\end{eqnarray*}
\end{lemma}
The next three lemmas give relations between the fundamental operators of a tetrablock contraction and its adjoint.
\begin{lemma}\label{tetralem4}
Let (A,B,P) be a tetrablock contraction on a Hilbert space $\mathcal{H}$ and $F_1, F_2$ and $G_1,G_2$ be fundamental operators of $(A,B,P)$ and $(A^*,B^*,P^*)$ respectively. Then
$$
PF_i=G_i^*P|_{\mathcal{D}_P}, \text{ for $i$=$1$ and $2$}.
$$
\end{lemma}
\begin{proof}
We shall prove only for $i=1$, the proof for $i=2$ is similar. Note that the operators on both sides are from $\mathcal{D}_P$ to $\mathcal{D}_{P^*}$.
Let $h,h' \in \mathcal{H}$ be any element. Then
\begin{eqnarray*}
&&\langle (PF_1-G_1^*P)D_Ph, D_{P^*}h' \rangle
\\
&=&
\langle D_{P^*}PF_1D_Ph,h' \rangle- \langle D_{P^*}G_1^*PD_{P}h,h' \rangle
\\
&=&
\langle P(D_P F_1D_P)h,h' \rangle - \langle (D_{P^*}G_1^*D_{P^*})Ph,h'\rangle
\\
&=&
\langle P(A-B^*P)h,h' \rangle - \langle (A-PB^*)Ph,h' \rangle
\\
&=&
\langle (PA-PB^*P-AP+PB^*P)h,h' \rangle =0.
\end{eqnarray*}
Hence the proof.
\end{proof}
\begin{lemma}\label{tetralem2}
Let $(A,B,P)$ be a tetrablock contraction on a Hilbert space $\mathcal{H}$ and $F_1,F_2$ and $G_1,G_2$ be fundamental operators of $(A,B,P)$ and $(A^*,B^*,P^*)$ respectively. Then
$$
D_PF_1=(AD_P-D_{P^*}G_2P)|{\mathcal{D}_P} \text{ and } D_PF_2=(BD_{P}-D_{P^*}G_1P)|{\mathcal{D_P}}.
$$
\end{lemma}
\begin{proof}
We shall prove only one of the above, proof of the other is similar. For $h \in \mathcal{H}$, we have
\begin{eqnarray*}
(AD_P-D_{P^*}G_2P)D_Ph
&=&
A(I-P^*P)h-(D_{P^*}G_2D_{P^*})Ph
\\
&=&
Ah-AP^*Ph-(B^*-AP^*)Ph
\\
&=&
Ah-AP^*Ph-B^*Ph+AP^*Ph
\\
&=&(A-B^*P)h=(D_P F_1)D_Ph.
\end{eqnarray*}
Hence the proof.
\end{proof}
\begin{lemma}\label{tetralem3}
Let $(A,B,P)$ be a tetrablock contraction on a Hilbert space $\mathcal{H}$ and $F_1,F_2$ and $G_1,G_2$ be fundamental operators of $(A,B,P)$ and $(A^*,B^*,P^*)$ respectively. Then
\begin{eqnarray*}
&&(F_1^*D_PD_{P^*}-F_2P^*)|_{\mathcal{D}_{P^*}}=D_PD_{P^*}G_1-P^*G_2^* \text{ and }
\\
&&(F_2^*D_PD_{P^*}-F_1P^*)|_{\mathcal{D}_{P^*}}=D_PD_{P^*}G_2-P^*G_1^*.
\end{eqnarray*}
\end{lemma}
\begin{proof}
For $h \in \mathcal{H}$, we have
\begin{eqnarray*}
(F_1^*D_PD_{P^*}-F_2P^*)D_{P^*}h
&=&
F_1^*D_P(I-PP^*)h-F_2P^*D_{P^*}h
\\
&=&
F_1^*D_Ph-F_1^*D_PPP^*h-F_2D_PP^*h
\\
&=&
F_1^*D_Ph-(F_1^*D_PP+F_2D_P)P^*h
\\
&=&
F_1^*D_Ph-D_PBP^*h\;\;\;\;\;\;\;\;\;\;\;\;[\text{by Lemma \ref{tetra}}]
\\
&=&
(AD_P-D_{P^*}G_2P)^*h-D_PBP^*h \;\;\;\;[\text{by Lemma \ref{tetralem2}}]
\\
&=&
D_PA^*h-P^*G_2^*D_{P^*}h-D_PBP^*h
\\
&=&
D_P(A^*-BP^*)h-P^*G_2^*D_{P^*}h
\\
&=&
D_PD_{P^*}G_1D_{P^*}h-P^*G_2^*D_{P^*}h
\\
&=&
(D_PD_{P^*}G_1-P^*G_2^*)D_{P^*}h.
\end{eqnarray*}
The other relation can be proved similarly.
\end{proof}

%In \cite{sir's tetrablock paper}, a tetrablock isometric dilation of a tetrablock contraction $(A,B,P)$ was explicitly constructed under the hypothesis that the fundamental operators $F_1,F_2$ of $(A,B,P)$ satisfy (\ref{Maa14}). Note that a tetrablock isometry is, by definition, the restriction of a tetrablock unitary. So once we have a tetrablock isometric dilation of a tetrablock contraction, we have a tetrablock unitary dilation too(since any extension of a dilation is again a dilation). But in the construction of the unitary dilation in this paper, we require two equations to hold, viz., (\ref{Maa14}) and $[G_1,G_2]=0$ and $[G_1,G_1^*]=[G_2,G_2^*]$, where $G_1,G_2$ are the fundamental operators of $(A^*,B^*,P^*)$. The next lemma shows that these two equations are equivalent.

{\bf{Observation}}
If $(N_1,N_2,N_3)$ is a commuting triple of bounded operators such that $N_3$ is unitary and $N_1=N_2^*N_3$, then $N_1$ and $N_2$ are normal operators.
\begin{proof}
$N_1=N_2^*N_3$ gives after multiplying $N_3^*$ from left $N_1N_3^*=N_2^*$, which after taking adjoint each side gives $N_2=N_3N_1^*=N_1^*N_3$, where last equality follows from Fuglede-Putnam's theorem(See \cite{conway functional analysis}).
$$ N_1N_1^*N_3=N_1N_2=N_2N_1=N_1^*N_3N_1=N_1^*N_1N_3. $$
Since $N_3$ is unitary, it follows that $N_1$ is a normal operator.
Also $$N_2N_2^*N_3=N_2N_1=N_1N_2=N_2^*N_3N_2=N_2^*N_2N_3.$$
Since $N_3$ is unitary, it follows that $N_2$ is normal operator.
\end{proof}

\section{Dilation Of A Tetrablock Contraction - Proof of Theorem \ref{1st major thm}}
We begin this section by showing that the fundamental operators $F_1$ and $F_2$ of a tetrablock contraction $(A,B,P)$ and the fundamental operators $G_1$ and $G_2$ of the adjoint tetrablock contraction $(A^*,B^*,P^*)$ are intimately related in the sense explained in the following lemma.
\begin{lemma}\label{brandnew}
Let $(A,B,P)$ be a tetrablock contraction on a Hilbert space $\mathcal{H}$ and $F_1,F_2$ and $G_1,G_2$ be fundamental operators of $(A,B,P)$ and $(A^*,B^*,P^*)$ respectively. Then
$$
[F_1, F_2] =0 \text{ and } [F_1, F_1^*] = [F_2, F_2^* ]$$ $$\text{if and only if}$$
$$ [G_1, G_2] =0 \text{ and } [G_1, G_1^*] = [G_2, G_2^* ].$$
\end{lemma}
\begin{proof}
Note that if we can show one of the above implications, then the other will follow from it. Because once we prove one implication, we then apply it for the tetrablock contraction $(A^*,B^*,P^*)$ to get the other implication. So we shall prove only one implication.

We first show that if $V$ on a Hilbert space $\tilde{\mathcal{H}}$ is the minimal isometric dilation of a contraction $P$ on Hilbert space $\mathcal{H}$, then dimensions of the spaces $\mathcal{D}_{V^*}$ and $\mathcal{D}_{P^*}$ are the same. Since $V$ is minimal, $\tilde{\mathcal{H}}=\bar{span}\{V^nh: h \in \mathcal{H}, n \geq 0 \}$. We have for all $h,h' \in \mathcal{H}$ and $n \geq 0$, $\langle V^*h, V^nh'\rangle= \langle h, V^{n+1}h' \rangle=\langle h, P^{n+1}h' \rangle=\langle P^*h, P^nh' \rangle = \langle P^*h, V^nh' \rangle$. Hence $V^*|_{\mathcal{H}}=P^*$. Note that $D_{V^*}^2V^nh=(I-VV^*)V^nh=0,$ for all $n \geq 1$. So the operator $D_{V^*}^2$ kills $\{V^nh, h \in \mathcal{H} \mbox{ and }n \geq 1\}$, so does the operator $D_{V^*}$. Therefore $\mathcal{D}_{V^*} = \bar{D_{V^*}\tilde{\mathcal{H}}} = \bar{D_{V^*}\mathcal{H}}$. Something more is true, $\lVert D_{V^*}h\rVert^2 = \langle (I-VV^*)h,h\rangle = \lVert h\rVert^2 - \lVert V^*h\rVert^2 = \lVert h\rVert^2 - \lVert P^*h\rVert^2 = \lVert D_{P^*}h\rVert^2$. So we define a unitary $X:\mathcal{D}_{P^*} \to \mathcal{D}_{V^*}$ by $XD_{P^*}h = D_{V^*}h$, for all $h \in \mathcal{H}$ and extend it to the closure continuously. That $X$ is a unitary, is clear from its very definition and from the fact that $\lVert D_{V^*}h\rVert=\lVert D_{P^*}h\rVert$, for all $h \in \mathcal{H}$. Note that the unitary $X$ satisfies $XD_{P^*}=D_{V^*}$, whenever $V$ is the minimal isometric dilation of $P$.

Let us suppose that the fundamental operators $F_1$ and $F_2$ of the tetrablock contraction $(A,B,P)$ satisfy $[F_1, F_2] =0 \text{ and } [F_1, F_1^*] = [F_2, F_2^* ]$. Then we know from \cite{sir's tetrablock paper}(Theorem 6.1) that $(A,B,P)$ has a tetrablock isometric dilation $(V_1,V_2,V_3)$ such that $V_3$ is the minimal isometric dilation of $P$ and in fact, $(V_1,V_2,V_3)$ is a co-extension of $(A,B,P)$. We shall prove that $[G_1, G_2] =0 \text{ and } [G_1, G_1^*] = [G_2, G_2^* ]$, where $G_1$ and $G_2$ are fundamental operators of $(A^*,B^*,P^*)$.

Now we show that $XG_1X^*$ and $XG_2X^*$ are the fundamental operators of the tetrablock co-isometry $(V_1^*,V_2^*,V_3^*)$, where $X$ is the unitary defined above. For $h,h' \in \mathcal{H}$ we have
\begin{eqnarray*}
&&\langle (A^*-BP^*)h,h' \rangle = \langle D_{P^*}G_1D_{P^*}h,h' \rangle
\\
&\Rightarrow& \langle A^*h,h' \rangle - \langle P^*h,B^*h' \rangle = \langle G_1D_{P^*}h,D_{P^*}h' \rangle
\\
&\Rightarrow&
\langle V_1^*h,h' \rangle - \langle V_3^*h,V_2^*h' \rangle = \langle G_1X^* D_{V_3^*}h,X^* D_{V_3^*}h' \rangle
\\
&\Rightarrow&
\langle (V_1^*-V_2V_3^*)h,h' \rangle = \langle D_{V_3^*}X G_1X^*D_{V_3^*}h,h' \rangle.
\end{eqnarray*}
Therefore $(V_1^*-V_2V_3^*) = D_{V_3^*}(X G_1X^*)D_{V_3^*}$. Similarly it can be showed that $(V_2^*-V_1V_3^*) = D_{V_3^*}(X G_2X^*)D_{V_3^*}$. By uniqueness of the fundamental operators, $XG_1X^*$ and $XG_2X^*$ are the fundamental operators of $(V_1^*,V_2^*,V_3^*)$.

Since $(V_1,V_2,V_3)$ is a tetrablock isometry on the Hilbert space $\mathcal{K}$, the space $\mathcal{K}$ can be decomposed into $\mathcal{K}_1 \oplus \mathcal{K}_2$ in such a way that $\mathcal{K}_1$ and $\mathcal{K}_2$ reduce $(V_1,V_2,V_3)$ and $(V_1,V_2,V_3)|_{\mathcal{K}_1} =: (\tilde{V_1},\tilde{V_2},\tilde{V_3})$ is pure tetrablock isometry and $(V_1,V_2,V_3)|_{\mathcal{K}_2}$ is tetrablock unitary. See \cite{sir's tetrablock paper}(Theorem 5.6) for a proof of this fact. Note that if the fundamental operators of the pure tetrablock co-isometry $(\tilde{V_1^*},\tilde{V_2^*},\tilde{V_3^*})$ are $\tilde{G_1}$ and $\tilde{G_2}$, then the fundamental operators of the tetrablock co-isometry $(V_1^*,V_2^*,V_3^*)$ are $ 0 \oplus \tilde{G_1}$ and $0 \oplus \tilde{G_2}$.

By Corollary 15 of \cite{sau}, we have that the pure tetrablock isometry $(\tilde{V_1},\tilde{V_2},\tilde{V_3})$ is unitarily equivalent to $(M_{\tilde{G_1}^* + \tilde{G_2}z}, M_{\tilde{G_2}^* + \tilde{G_1}z}, M_z)$, which acts on the Hilbert space $H^2(\mathcal{D}_{\tilde{V_3}^*})$. By commutativity of the triple $(M_{\tilde{G_1}^* + \tilde{G_2}z}, M_{\tilde{G_2}^* + \tilde{G_1}z}, M_z)$ we get
$$
[\tilde{G}_1, \tilde{G}_2] =0 \text{ and } [\tilde{G}_1, \tilde{G}_1^*] = [\tilde{G}_2, \tilde{G}_2^* ].
$$
Hence the fundamental operators of the tetrablock co-isometry $(V_1^*,V_2^*,V_3^*)$, which are nothing but $ 0 \oplus \tilde{G_1}$ and $0 \oplus \tilde{G_2}$ also satisfy the above equality. But the fundamental operators of the tetrablock co-isometry $(V_1^*,V_2^*,V_3^*)$, as observed above, are $XG_1X^*$ and $XG_2X^*$, where recall that $X$ is a unitary. Hence $G_1$ and $G_2$ also satisfy the above equality. In other words $G_1$ and $G_2$ satisfy
$$
[G_1, G_2] =0 \text{ and } [G_1, G_1^*] = [G_2, G_2^* ].$$
Hence the proof.
\end{proof}

We now prove that $(R_1,R_2,U)$ on $\mathcal{K}$ defined in the statement of Theorem \ref{1st major thm} is a tetrablock unitary. To prove that, we shall show the following
\begin{enumerate}
\item[(i)] $R_1,R_2$ and $U$ commute with each other,
\item[(ii)] $R_1=R_2^*U$ and
\item[(iii)] $R_2$ is a contraction.
\end{enumerate}
These will imply that $(R_1,R_2,U)$ is a tetrablock unitary by
part (2) of Theorem \ref{tetrathm1}.

\underline{\bf{Proof of part (i)}}
It can be easily checked that the operators $D_1,D_2$ and $D_3$ on $l^2(\mathcal{D}_{P^*})$ defined in Theorem \ref{1st major thm} are unitarily equivalent to the operators $M^*_{G_1+G_2^*z}, M^*_{G_2+G_1^*z}$ and $M_z^*$ on $H^2_{\mathcal{D}_{P^*}}(\mathbb{D})$ respectively.
To show that $ R_1= \left(
     \begin{array}{cc}
       V_1 & C_1 \\
       0 & D_1 \\
     \end{array}
   \right)
$ and
$
R_2=
\left(
              \begin{array}{cc}
                V_2 & C_2 \\
                0 & D_2 \\
              \end{array}
            \right)
$ commute, we shall have to show
$V_1V_2=V_2V_1,$ $D_1D_2=D_2D_1$ and $V_1C_2+C_1D_2=V_2C_1+C_2D_1$.
\\
That $V_1$ and $V_2$ commute follows from the fact that $(V_1,V_2,V_3)$ is a tetrablock isometry.
%Note that, under the Hilbert space isomorphism $W^*$, as in section 3, which sends $l^2(\mathcal{D}_{P^*})$ to $H^2(\mathcal{D}_{P^*})$, the
%operators $D_1$, $D_2$ and $D_3$ go to the operators $M_{G_1+zG_2^*}^*$, $M_{G_2+zG_1^*}^*$ and $M_z^*$ respectively. Therefore, to show that $D_1$ and $D_2$ commute it is equivalent to show that $M_{G_1+zG_2^*}$ and $M_{G_2+zG_1^*}$ commute. They commute if and only if $(G_1+zG_2^*)(G_2+zG_1^*)=(G_2+zG_1^*)(G_1+zG_2^*)$, for all $z \in \mathbb{D}$, which hold if and only if $[G_1,G_2]=0$ and $[G_1,G_1^*]=[G_2,G_2^*]$ hold, which is part of the assumptions. Hence $D_1$ and $D_2$ commute.
%\\
%\begin{eqnarray*}
%&&D_1D_2(a_0,a_1,a_2,\dots)
%\\
%&=&
%D_1(G_1^*a_0+G_2a_1,G_1^*a_1+G_2a_2,G_1^*a_2+G_2a_3,\dots)
%\\
%&=&
%(G_1^*G_2^*a_0+(G_1^*G_1+G_2G_2^*)a_1+G_2G_1a_2,G_1^*G_2^*a_1+(G_1^*G_1+G_2G_2^*)a_2+G_2G_1a_3,
%\\
%&&G_1^*G_2^*a_2+(G_1^*G_1+G_2G_2^*)a_3+G_2G_1a_4,\dots)
%\end{eqnarray*}
%\begin{eqnarray*}
%&&D_2D_1(a_0,a_1,a_2,\dots)
%\\
%&=&
%D_2(G_2^*a_0+G_1a_1,G_2^*a_1+G_1a_2,G_2^*a_2+G_1a_3,\dots)
%\\
%&=&
%(G_2^*G_1^*a_0+(G_2^*G_2+G_1G_1^*)a_1+G_1G_2a_2,G_2^*G_1^*a_1+(G_2^*G_2+G_1G_1^*)a_2+G_1G_2a_3,
%\\
%&&G_2^*G_1^*a_2+(G_2^*G_2+G_1G_1^*)a_3+G_1G_2a_4,\dots)
%\end{eqnarray*}
As we have observed, the commutativity of $D_1$ and $D_2$ is equivalent to that of $M_{G_2+G_1^*z}$ and $M_{G_1+G_2^*z}$. It can be easily checked that the commutativity of $M_{G_2+G_1^*z}$ and $M_{G_1+G_2^*z}$ is equivalent to $G_1, G_2$ satisfying equation (\ref{Maa14}) in place of $F_1$ and $F_2$ respectively, which holds true. So we have $D_1D_2=D_2D_1.$
From the definition of the operators, it can be easily checked that
\begin{eqnarray}\label{app}V_1C_2+C_1D_2=V_2C_1+C_2D_1.\end{eqnarray}
The detailed proof of this can be found in the Appendix.

\underline{\bf{Proof of part (ii)}}
To prove $R_1=R_2^*U$ it is equivalent to show the following:
$$
V_1=V_2^*V_3,\;C_1=V_2^*C_3,\;C_2^*V_3=0 \text{ and } D_1=C_2^*C_3+D_2^*D_3
$$
We shall check these conditions one by one.
Since $(V_1,V_2,V_3)$ is a tetrablock isometry, the first condition is satisfied.
\begin{eqnarray*}
V_2^*C_3(a_0,a_1,a_2,\dots)
&=&
V_2^*(D_{P^*}a_0 \oplus (-P^*a_0,0,0,\dots))
\\
&=&
(B^*D_{P^*}-D_P F_1P^*)a_0\oplus(-F_2^*P^*a_0,0,0,\dots)
\\
&=&C_1(a_0,a_1,a_2,\dots) \; [\text{using Lemma \ref{tetralem2}}].
\end{eqnarray*}
\begin{eqnarray*}
C_2^*V_3(h \oplus (a_0,a_1,a_2,\dots))
&=&
C_2^*(Ph \oplus (D_Ph, a_0, a_1, \dots ))
\\
&=&
((G_1^*D_{P^*}P-PF_1D_P)a_0,0,0,\dots)
\\
&=&(0,0,0,\dots) \; [\text{using Lemma \ref{tetralem4}}].
\end{eqnarray*}
\begin{eqnarray*} \text{ Now }
&&(C_2^*C_3+D_2^*D_3)(a_0,a_1,a_2,\dots)
\\
&=&
C_2^*(D_{P^*}a_0 \oplus (-P^*a_0,0,0,\dots)) + D_2^*(a_1,a_2,a_3,\dots)
\\
&=&
((G_1^*D_{P^*}^2+PF_1P^*)a_0,0,0,\dots) + (G_2a_1,G_1^*a_1+G_2a_2,G_1^*a_2+G_2a_3,\dots)
\\
&=&
((G_1^*-G_1^*PP^*+PF_1P^*)a_0+G_2a_1,G_1^*a_1+G_2a_2,G_1^*a_2+G_2a_3,\dots)
\\
&=&
(G_1^*a_0+G_2a_1,G_1^*a_1+G_2a_2,G_1^*a_2+G_2a_3,\dots)
\\
&=&D_1(a_0,a_1,a_2,\dots)\;[\text{using Lemma \ref{tetralem4}}].
\end{eqnarray*}
So this completes the proof of part (ii).

\underline{\bf{Proof of part (iii)}}
First we shall show that $r(D_2) \leq 1$. What we actually show is that, numerical radius of $D_2$ is not greater than one. Since spectral radius is not greater than the numerical radius (See \cite{Rosenthal},Theorem 1.2.11. ),
we shall be done.
Let us define
$$
\varphi :\mathbb{D} \to  \mathcal{B}(\mathcal{D}_{P^*}) \text{ by }
\varphi(z)=G_2 + zG_1^*.
$$
Clearly $\varphi$ is holomorphic, bounded and continuous on the boundary $\partial\mathbb{D} = \mathbb{T}$ of the disk.
As observed before in the introduction, the operator $D_2^*$ goes to multiplication by the function $\varphi$.
Now $w(M_{\varphi}) \leq  \text{sup}\{w(\varphi(z)) : z \in \mathbb{T}\}$. Recall that the numerical radius of an operator $X$ is
not greater than one if and only if the real part of the operator zX is not bigger than
identity for every $z$ on the unit circle, see \cite{numerical radius}.
From Theorem 3.5 of \cite{sir's tetrablock paper}, we know that $w(G_1+zG_2)\leq 1$, which in turn gives that $w(z_1G_1+z_2G_2)\leq 1$, for all $z_1,z_2$ in the unit circle. Which is equivalent to $(z_1G_1+z_2G_2)+(z_1G_1+z_2G_2)^* \leq 2I$, which gives after rearranging $(z_2G_2+\bar{z}_1G_1^*)+(z_2G_2+\bar{z}_1G_1^*)^* \leq 2I$,
which is equivalent to saying that $z_2(G_2+zG_1^*)+\bar{z}_2(G_2+zG_1^*)^* \leq 2I$, for every $z$ and $z_2$ on the unit circle. Which implies that $w(G_2+zG_1^*) \leq 1$, for al $z$ in the unit circle. Hence $w(M_{\varphi})\leq 1$. This implies that $r(D_2)=r(M_{\varphi}) \leq w(M_{\varphi}) \leq 1$.
Since $R_2$ of the form
$
\begin{pmatrix}
V_2 & C_2 \\
0 & D_2
\end{pmatrix}$, we have by Lemma 1 of \cite{Perturbation of spectrum}, that $\sigma(R_2)\subseteq \sigma(V_2)\cup \sigma (D_2)$. Since $r(V_2) \leq 1$ ($(V_1,V_2,V_3)$ being tetrablock isometry) and $r(D_2) \leq 1$, we have $r(R_2) \leq 1$. Since $R_2$ is a normal operator(applying the observation in previous section to the triple $(R_1,R_2,U)$), we have by Stampfli's theorem (which says that, if $X$ is hyponormal, then $||X^n||=||X||^n$ and so $||X||=r(X)$, see Proposition 4.9 of \cite{stampfli}) that $||R_2|| \leq 1$. This completes the proof of part (iii).

Therefore $(R_1,R_2,U)$ is a tetrablock unitary. Now to prove
$(R_1,R_2,U)$ is a dilation of $(A,B,P)$, we observe that each of
$R_1,R_2$ and $U$ are upper triangular with respect to the
decomposition $\tilde{\mathcal{H}} \oplus l^2(\mathcal{D}_{P^*})$
of $\mathcal{K}$ with $V_1,V_2$ and $V_3$ in the $(11)$-th places of
the operator matrices respectively. Also noting that each of
$V_1,V_2$ and $V_3$ are lower triangular with respect to the
decomposition $\mathcal{H} \oplus l^2(\mathcal{D}_P)$ of
$\tilde{\mathcal{H}}$ with $A,B$ and $P$ in the $(11)$-th places of
the operator matrices respectively, we get
$$A^mB^nP^l=P_\mathcal{H}R_1^mR_2^nU^l|_\mathcal{H},\;\text{for all $m,n,l \geq 0.$}$$
This completes the proof of Theorem \ref{1st major thm}. \qed

\begin{remark}[Minimality]
Minimality of a commuting normal boundary dilation $\underline{N} = (N_1, N_2, \ldots , N_d)$ on a space $\clk$ of a commuting tuple $(T_1, T_2, \ldots ,T_d)$ of bounded operators on a space $\clh$ means that the space $\clk$ is no bigger than
$$\overline{span} \{ N_1^{k_1} N_2^{k_2} \ldots N_d^{k_d} N_1^{*l_1} N_2^{*l_2} \ldots N_d^{*l_d} h : h \in \clh \mbox{ where } k_i \mbox{ and } l_i \in \mathbb{N} \mbox{ for } i=1,2,\ldots, d \}.$$ In our construction, the space is the minimal unitary dilation space of the contraction $P$, which is obviously a subspace of $\overline{span} \{ R_1^{k_1}R_2^{k_2}{R^*}_1^{l_1}{R^*}_2^{l_2}U^nh: h \in \mathcal{H}, k_i, l_i \geq 0 \text{ and } n \in \mathbb{Z}$\}. Note that no dilation of $(A,B,P)$ can take place on a space smaller than the minimal unitary dilation space of the contraction $P$, since the dilation has to dilate $P$ also. Hence the dilation is minimal.
\end{remark}

\section{Uniqueness of the dilation}
In this section, we show that the minimal dilation $(R_1,R_2,U)$ of $(A,B,P)$ defined in Theorem \ref{1st major thm} is unique under a suitable condition.

Note that the operator $U$, when written with respect to the decomposition $l^2(\mathcal{D}_P) \oplus \mathcal{H} \oplus l^2(\mathcal{D}_{P^*})$, takes the following form,
$$
\begin{pmatrix}
U_1 & U_2 & U_3 \\
0   &  P  & U_4 \\
0   &  0  & U_5 \\
\end{pmatrix},
$$
where $U_1, U_2, U_3, U_4$ and $U_5$ are as follows
\begin{eqnarray*}
&&U_1(a_0,a_1,a_2,\dots)=(0,a_0,a_1,\dots),\; U_2(h)=(D_ph,0,0,\dots)
\\
&&U_3(b_0,b_0,b_2,\dots)=(-P^*b_0,0,0,\dots),\; U_4(b_0,b_1,b_2,\dots)=D_{P^*}b_0 \text{ and}
\\
&&U_5(b_0,b_1,b_2,\dots)=(b_1,b_2,b_3,\dots),
\end{eqnarray*}
for all $h \in \mathcal{H}, (a_0,a_1,a_2,\dots) \in l^2(\mathcal{D}_P)$ and $(b_0,b_0,b_2,\dots) \in l^2(\mathcal{D}_{P^*})$. Note that this is the Sch$\ddot{\text{a}}$ffer minimal unitary dilation of the contraction $P$(See \cite{sfr} or ch. 1, sec. 5 in \cite{Nagy-Foias}).

The next result says that if $(R_1,R_2,\dots,R_{n-1},U)$ is a dilation of $(S_1,S_2,\dots,S_{n-1},P)$, where $P$ is a contraction and $U$ is the Sch$\ddot{\text{a}}$ffer minimal unitary dilation of $P$, then the operators $R_j$, $j=1,2,\dots, n-1$ can not be of arbitrary form.
\begin{lemma}\label{added lemma}
Let $(R_1,R_2,\dots,R_{n-1},U)$ on $\mathcal{K}$ be a dilation of $(S_1,S_2,\dots,S_{n-1},P)$ on $\mathcal{H}$, where $P$ is a contraction on $\mathcal{H}$ and $U$ on $\mathcal{K}$ is the Sch$\ddot{\text{a}}$ffer minimal unitary dilation of $P$. Then for all $j=1,2,\dots , n-1$, $R_j$ admits a matrix representation of the form
$$
\begin{pmatrix}
* & * & * \\
0 & S_j & * \\
0 & 0 & * \\
\end{pmatrix},
$$
with respect to the decomposition $\mathcal{K}=l^2(\mathcal{D}_P) \oplus \mathcal{H} \oplus l^2(\mathcal{D}_{P^*})$.
\end{lemma}
\begin{proof}
Let $R_j=(R^j_{kl})_{k,l=1}^3$ with respect to $\mathcal{K}=l^2(\mathcal{D}_P) \oplus \mathcal{H} \oplus l^2(\mathcal{D}_{P^*})$ for each $j=1,2,\dots,n-1$. Call $\tilde{\mathcal{H}}=l^2(\mathcal{D}_P) \oplus \mathcal{H}$. Since $U$ is minimal we have $\mathcal{K}=\bigvee_{m=-\infty}^{\infty}U^m\mathcal{H}$ and $\tilde{\mathcal{H}}=\bigvee_{m=0}^{\infty}U^m\mathcal{H}=\bigvee_{m=0}^{\infty}V^m\mathcal{H}$, where $V$ is the minimal isometry dilation of $P$.
Note that
$$
P_{\mathcal{H}}R_j(U^mh)=S_jP^mh= S_jP_{\mathcal{H}}U^mh, \text{ for all $h \in \mathcal{H}, m \in \mathbb{N}$ and $j=1,2,\dots, n-1$.}
$$
Hence we have $P_{\mathcal{H}}R_j|_{\tilde{\mathcal{H}}}=S_jP_{\mathcal{H}}|_{\tilde{\mathcal{H}}}$ or equivalently $S_j^*=P_{\tilde{\mathcal{H}}}R_j^*|_{\mathcal{H}}$ for all $j=1,2,\dots, n-1$. This shows that $R^j_{21}=0$, for all $j=1,2,\dots,n-1$.
\\
Call $\tilde{\mathcal{N}}=\mathcal{H} \oplus l^2(\mathcal{D}_{P^*})$, then note that $\tilde{\mathcal{N}}=\bigvee_{n=0}^{\infty}{U^*}^n\mathcal{H}$.
We have
$$
P_{\mathcal{H}}R^*_j({U^*}^mh)=S^*_j{P^*}^mh= S^*_jP_{\mathcal{H}}{U^*}^mh, \text{ for all $h \in \mathcal{H}, m \in \mathbb{N}$ and $j=1,2,\dots, n-1$.}
$$
This and a similar argument as above give us $S_j=P_{\tilde{\mathcal{N}}}R_j|_{\mathcal{H}}$. Therefore $R^j_{32}=0$, for all $j=1,2,\dots,n-1$.
\\
So far, we have showed that for each $j=1,2,\dots, n-1$, $R_j$ admits the matrix representation of the form
$$
\begin{pmatrix}
R^j_{11} & R^j_{12} & R^j_{13} \\
0 & S_j & R^j_{23} \\
R^j_{31} & 0 & R^j_{33} \\
\end{pmatrix},
$$
with respect to the decomposition $\mathcal{K}=l^2(\mathcal{D}_P) \oplus \mathcal{H} \oplus l^2(\mathcal{D}_{P^*})$. To show that $R^j_{13}=0$ we proceed as follows.
From the commutativity of $R_j$ with $U$ we get, by an easy matrix calculation
\begin{eqnarray}\label{addition}
R^j_{31}U_1=U_5R^j_{31} \text{ and } R^j_{31}U_2=0,
\end{eqnarray}
(equating the $31^{th}$ entries and $32^{th}$ entries of $R_jU$ and $UR_j$ respectively).
By the definition of $U_2$, we have $RanU_2=Ran(I-U_1U_1^*)$. Therefore $R^j_{31}(I-U_1U_1^*)=0$. Which with the first equation of (\ref{addition}) gives
$R^j_{31} = U_5R^j_{31}U_1^*$. Which gives after n-th iteration $R^j_{31}=U_5^nR^j_{31}{U_1^*}^n$. Now since ${U_1^*}^n \to 0$ as $n \to \infty$, we have that $R^j_{31}=0$ for each $j=1,2,\dots, n-1$. This completes the proof of the lemma.
\end{proof}

First we prove a weaker version of the uniqueness theorem, which will be used to prove the stronger version of the uniqueness theorem.
\begin{lemma}\label{Uniqueness1}
Suppose $(A,B,P)$ is a tetrablock contraction on a Hilbert space $\mathcal{H}$ and $(R_1,R_2,U)$ is the above tetrablock unitary dilation of $(A,B,P)$. If $(\tilde{R}_1,\tilde{R}_2,U)$ is another tetrablock unitary dilation of $(A,B,P)$ such that $\tilde{R}_1$ and $\tilde{R}_2$ are extensions of $V_1$ and $V_2$ respectively, then $\tilde{R}_1=R_1$ and $\tilde{R}_2=R_2$.
\end{lemma}
\begin{proof}Since $\tilde{R}_1$ and $\tilde{R}_2$ on the Hilbert space $\mathcal{K}$ are such that they are extensions of $V_1$ and $V_2$ respectively, the matrix representations of $\tilde{R_1}$ and $\tilde{R_2}$ with respect to the decomposition $\mathcal{K}=\tilde{H} \oplus l^2(\mathcal{D}_{P^*})$ will be of the form
$
\begin{pmatrix}
V_1 & \tilde{\tilde{C}}_1 \\
0 & \tilde{D}_1 \\
\end{pmatrix}
$ and
$
\begin{pmatrix}
V_2 & \tilde{C}_2 \\
0 & \tilde{D}_2 \\
\end{pmatrix}
$
respectively, where $\tilde{C}_1,\tilde{C}_2:l^2({\mathcal{D}_{P^*}}) \to \tilde{\mathcal{H}}$ and $\tilde{D}_1,\tilde{D}_2:l^2({\mathcal{D}_{P^*}}) \to l^2({\mathcal{D}_{P^*}})$ are some operators. We want to show that $\tilde{C}_1,\tilde{C}_2,\tilde{D}_1$ and $\tilde{D}_2$ are same as $C_1,C_2,D_1$ and $D_2$ respectively.
Since $(\tilde{R}_1,\tilde{R}_2,U)$ is a tetrablock unitary, we have by part (2) of Theorem \ref{tetrathm1} the following:
$\tilde{R_1},\tilde{R_2}$  and unitary operator $U$ commute, $\tilde{R_2}$ is a contraction, and $\tilde{R_1}=\tilde{R_2}^*U.$
The fact that $U$ is unitary, gives us
\begin{eqnarray}\label{tetra1}
D_3^*D_3+C_3^*C_3=I \text{ and }C_3^*V_3=0.
\end{eqnarray}
The fact that $\tilde{R_2}$ and $U$ commute, gives us
\begin{eqnarray}\label{tetra23}
V_2C_3+\tilde{C}_2D_3=V_3\tilde{C}_2+C_3\tilde{D}_2 \text{ and } \tilde{D}_2D_3=D_3\tilde{D}_2.
\end{eqnarray}
The fact that $\tilde{R_1}$ and $U$ commute, gives us
\begin{eqnarray}\label{tetra13}
V_1C_3+\tilde{C}_1D_3=V_3\tilde{C}_1+C_3\tilde{D}_1 \text{ and } \tilde{D}_1D_3=D_3\tilde{D}_1,
\end{eqnarray}
and commutativity of $\tilde{R_1}$ and $\tilde{R_2}$ gives
\begin{eqnarray}\label{tetra12}
V_1\tilde{C}_2+\tilde{C}_1\tilde{D}_2=V_2\tilde{C}_1+\tilde{C}_2\tilde{D}_1 \text{ and } \tilde{D}_1\tilde{D}_2=\tilde{D}_2\tilde{D}_1.
\end{eqnarray}
Since $\tilde{R_1}=\tilde{R_2}^*U$, we have
\begin{eqnarray}\label{tetra123}
\tilde{C}_1=V_2^*C_3 \text{ and } \tilde{D}_1=\tilde{C}_2^*C_3+\tilde{D}_2^*D_3.
\end{eqnarray}
Therefore
\begin{eqnarray*}
\tilde{C}_1(a_0,a_1,a_2,\dots)
&=&
V_2^*C_3(a_0,a_1,a_2,\dots)
\\
&=&
V_2^*(D_{P^*}a_0 \oplus (-P^*a_0,0,0,\dots))
\\
&=&
((B^*D_{P^*}-D_PF_1P^*)a_0 \oplus (-F_2^*P^*a_0,0,0,\dots))
\\
&=&
(D_{P^*}G_2a_0 \oplus (-F_2^*P^*a_0,0,0,\dots))\;[\text{by Lemma \ref{tetralem2}}]
\\
&=&
C_1(a_0,a_1,a_2,\dots).
\end{eqnarray*}
Also we have $\tilde{R_2}=\tilde{R_1}^*U$, which gives
\begin{eqnarray}\label{tetra213}
\tilde{C}_2=V_1^*C_3 \text{ and } \tilde{D}_2=\tilde{C}_1^*C_3+\tilde{D}_1^*D_3.
\end{eqnarray}
Therefore
\begin{eqnarray*}
\tilde{C}_2(a_0,a_1,a_2,\dots)
&=&
V_1^*C_3(a_0,a_1,a_2,\dots)
\\
&=&
V_1^*(D_{P^*}a_0 \oplus (-P^*a_0,0,0,\dots))
\\
&=&
((A^*D_{P^*}-D_PF_2P^*)a_0 \oplus (-F_1^*P^*a_0,0,0,\dots))
\\
&=&
(D_{P^*}G_1a_0 \oplus (-F_1^*P^*a_0,0,0,\dots))\;[\text{by Lemma \ref{tetralem2}}]
\\
&=&
C_2(a_0,a_1,a_2,\dots).
\end{eqnarray*}
To find $\tilde{D}_1$ and $\tilde{D}_2$, we proceed by multiplying (\ref{tetra23}) by $C_3^*$ from left. Then using (\ref{tetra1}), we get
$$
\tilde{D}_2^*(I-D_3^*D_3)=D_3^*\tilde{C}_2^*C_3+C_3^*V_2^*C_3.
$$
Noting that $(I-D_3^*D_3)$ is the orthogonal projection of $l^2(\mathcal{D}_{P^*})$ onto the first component, we get
\begin{eqnarray*}
&&\tilde{D}_2^*(a_0,0,0,\dots)
\\
&=&
(D_3^*\tilde{C}_2^*C_3+C_3^*V_2^*C_3)(a_0,a_1,a_2,\dots)
\\
&=&
D_3^*\tilde{C}_2^*(D_{P^*}a_0 \oplus (-P^*a_0,0,0,\dots))+C_3^*V_2^*(D_{P^*}a_0 \oplus (-P^*a_0,0,0,\dots))
\\
&=&
D_3^*((G_1^*(I-PP^*)+P F_1P^*)a_0,0,0,\dots) \\
&& + \; C_3^*((B^*D_{P^*}-D_P F_1P^*)a_0 \oplus (-F_2^*P^*a_0,0,0,\dots))
\\
&=&
(0,(G_1^*(I-PP^*)+P F_1P^*)a_0,0,0,\dots) \\
&& + \; ((D_{P^*}B^*D_{P^*}-D_{P^*}D_P F_1P^*+P F_2^*P^*)a_0,0,0,\dots)
\\
&=&
((D_{P^*}(B^*D_{P^*}-D_P F_1P^*)+P F_2^*P^*)a_0,(G_1^*(I-PP^*)+P F_1P^*)a_0,0,0,\dots)
\\
&=&
(((I-PP^*)G_2+P F_2^*P^*)a_0,(G_1^*(I-PP^*)+P F_1P^*)a_0,0,0,\dots) \; [\text{ by Lemma \ref{tetralem2}}]
\\
&=&
(G_2a_0,G_1^*a_0,0,0,\dots) \; [\text{ by Lemma \ref{tetralem4}}]
\end{eqnarray*}
Therefore for $a\in\mathcal{D}_{P^*}$, we have
$
\tilde{D}_2^*(a,0,0,\dots)=(G_2a,G_1^*a,0,0,\dots).
$
\begin{eqnarray*}
&&\tilde{D}_2^*(\overbrace{0,\dots,0}^\text{$n$ times},a,0,\dots )
\\
&=& \tilde{D}_2^*{D_3^*}^n(a,0,0,0,\dots)
\\
&=&
{D_3^*}^n\tilde{D}_2^*(a,0,0,0,\dots) [\text{ using last equation of (\ref{tetra23})}]
\\
&=&
{D_3^*}^n(G_2a,G_1^*a,0,0,\dots)
\\
&=&
(\overbrace{0,\dots,0}^{\text{$n$ times}},G_2a,G_1^*a,0,0,\dots), \text{ for every $n \geq 0.$}.
\end{eqnarray*}
Therefore for an arbitrary element $(a_0,a_1,a_2,\dots) \in l^2(\mathcal{D}_{P^*})$, we have
\begin{eqnarray*}
\tilde{D}_2^*(a_0,a_1,a_2,\dots)&=&\tilde{D}_2^*((a_0,0,0,\dots) + (0,a_1,0,\dots) + (0,0,a_2,\dots) +\cdots )
\\
&=&
(G_2a_0,G_1^*a_0,0,0,\dots) \\
&& + \;  (0,G_2a_1,G_1^*a_1,0,0,\dots) + (0,0,G_2a_2,G_1^*a_2,0,0,\dots) + \cdots
\\
&=&
(G_2a_0,G_1^*a_0+G_2a_1,G_1^*a_1+G_2a_2,\dots).
\end{eqnarray*}
For $(a_0,a_1,a_2,\dots)$ and $(b_0,b_1,b_2,\dots)$ in $l^2(\mathcal{D}_{P^*})$, we have
\begin{eqnarray*}
&&\langle(a_0,a_1,a_2,\dots),\tilde{D}_2^*(b_0,b_1,b_2,\dots) \rangle
\\
&=& \langle (a_0,a_1,a_2,\dots), (G_2b_0,G_1^*b_0+G_2b_1,G_1^*b_1+G_2b_2,\dots) \rangle
\\
&=&
\langle a_0,G_2 b_0 \rangle + \langle a_1,G_1^*b_0+G_2 b_1 \rangle + \langle a_2,G_1^*b_1+G_2 b_2 \rangle + \cdots
\\
&=&
\langle G_2^*a_0+G_1a_1,b_0 \rangle + \langle G_2^*a_1+G_1a_2,b_1 \rangle + \langle G_2^*a_2+G_1a_3, b_2 \rangle + \cdots
\\
&=&
\langle (G_2^*a_0+G_1a_1,G_2^*a_1+G_1a_2,G_2^*a_2+G_1a_3,\dots), (b_0,b_1,b_2,\dots) \rangle.
\end{eqnarray*}
Hence by definition of adjoint of an operator we have
\\$\tilde{D}_2(a_0,a_1,a_2,\dots)= (G_2^*a_0+G_1a_1,G_2^*a_1+G_1a_2,G_2^*a_2+G_1a_3,\dots)=D_2(a_0,a_1,a_2,\dots)$, for every $(a_0,a_1,a_2,\dots) \in l^2(\mathcal{D}_{P^*})$.
Therefore $\tilde{R_2}=R_2$.
\\
Similarly, multiplying (\ref{tetra13}) by $C_3^*$ from left, using (\ref{tetra1}) and proceeding in the same way as above one gets $\tilde{D}_1$ to be
$$
\tilde{D}_1(a_0,a_1,a_2,\dots)=(G_1^*a_0+G_2a_1,G_1^*a_1+G_2a_2,G_1^*a_2+G_2a_3,\dots)=D_1(a_0,a_1,a_2,\dots).
$$
Hence $\tilde{R_1}=R_1$.
This completes the proof.
\end{proof}
Now we are ready to state and prove the stronger version of the uniqueness theorem, the main result of this section.
\begin{theorem}[Uniqueness]\label{uniqueness}
Let $(A,B,P)$ be a tetrablock contraction on a Hilbert space $\mathcal{H}$ and $(R_1,R_2,U)$ as defined in Theorem \ref{1st major thm}, be the tetrablock unitary dilation of $(A,B,P)$.
\begin{enumerate}
\item[(i)] If $(\tilde{R}_1,\tilde{R}_2,U)$ is another tetrablock unitary dilation of $(A,B,P)$, then $\tilde{R}_1=R_1$ and $\tilde{R}_2=R_2$.
\item[(ii)] If $(\tilde{R}_1,\tilde{R}_2,\tilde{U})$ on some Hilbert space $\tilde{\mathcal{K}}$ containing $\mathcal{H}$, is another tetrablock unitary dilation of $(A,B,P)$ where $\tilde{U}$ is a minimal unitary dilation of $P$, then $(\tilde{R}_1,\tilde{R}_2,\tilde{U})$ is unitarily equivalent to $(R_1,R_2,U)$.
\end{enumerate}
\end{theorem}
\begin{proof}
(i) Since $(\tilde{R}_1,\tilde{R}_2,U)$ is a dilation of $(A,B,P)$, by Lemma \ref{added lemma} $\tilde{R_1}$ and $\tilde{R_2}$ are of the form
$$
\left(
  \begin{array}{cc}
   T_1  & \tilde{R'}_{12} \\
    0 & \tilde{R'}_{22} \\
  \end{array}
  \right)
  \text{ and }
\left(
  \begin{array}{cc}
   T_2  & \tilde{R''}_{12} \\
    0 & \tilde{R''}_{22} \\
  \end{array}
  \right) \text{ respectively,}
$$
with respect to the decomposition $\tilde{\mathcal{H}} \oplus l^2(\mathcal{D}_{P^*})$. Where $T_1$ and $T_2$ are operators on $\tilde{\mathcal{H}}$, which admit the matrix representation
$$
\left(
  \begin{array}{cc}
   T'_{11}  & T'_{12} \\
    0 & A \\
  \end{array}
  \right) \text{ and }
\left(
  \begin{array}{cc}
   T''_{11}  & T''_{12} \\
    0 & B \\
  \end{array}
  \right) \text{ respectively,}
$$
with respect to the decomposition $l^2(\mathcal{D}_P) \oplus \mathcal{H}  $. Since $(T_1,T_2,V)$ on $\tilde{H}$ is the restriction of the tetrablock contraction $(\tilde{R}_1,\tilde{R}_2,U)$ to $\tilde{\mathcal{H}}$ and $V$ is an isometry, we have $(T_1,T_2,V)$ a tetrablock isometry. Also note that ${T_1}^*|_{\mathcal{H}}=A^*$, $T_2^*|_{\mathcal{H}}=B^*$ and $V^*|_{\mathcal{H}}=P^*$. So $(T_1,T_2,V)$ is a tetrablock isometric dilation of $(A,B,P)$, where $V$ is the Sch$\ddot{\text{a}}$ffer minimal isometric dilation of $P$. Now it follows from the argument given in the proof of Theorem 6.1(2) of \cite{sir's tetrablock paper} that $T_1=V_1$ and $T_2=V_2$, where $V_1$ and $V_2$ are as in Theorem \ref{1st major thm}. Therefore $\tilde{R_1}$ and $\tilde{R_2}$ are extensions of $V_1$ and $V_2$ respectively. Now the proof follows from Lemma \ref{Uniqueness1}.

(ii) Since $\tilde{U}$ is a minimal unitary dilation of $P$, there exists a unitary operator $W: \tilde{\mathcal{K}} \to \mathcal{K}$ such that $W\tilde{U}W^*=U$ and $Wh=h$ for all $h \in \mathcal{H}$. This shows that $(W\tilde{R_1}W^*,W\tilde{R_2}W^*,W\tilde{U}W^*)$ is another tetrablock unitary dilation of $(A,B,P)$. But $W\tilde{U}W^*=U$. Hence by part (i) we have $(W\tilde{R_1}W^*,W\tilde{R_2}W^*,W\tilde{U}W^*)=(R_1,R_2,U)$. This completes the proof of part (ii).
\end{proof}

We conclude with the following remark.
\begin{remark}
 The tetrablock unitary dilation we have constructed in this section, is simple because it acts on a small space, viz., the minimal unitary dilation space of the contraction $P$. Moreover, we have showed that if the unitary part of the tetrablock unitary dilation is the minimal unitary dilation of the contraction $P$, then the tetrablock unitary dilation is unique upto unitary equivalence.
\end{remark}

\section{Appendix}
Here we prove that $V_1C_2+C_1D_2=V_2C_1+C_2D_1$, as in equation (\ref{app}).
\\From the definition of the operators $V_1,V_2$ and $V_3$ on $\tilde{\mathcal{H}}$, it is easy to see that they have the following matrix forms with respect to the decomposition
$\mathcal{H} \oplus \mathcal{D}_{P^*}\oplus \mathcal{D}_{P^*}\oplus \mathcal{D}_{P^*} \oplus \cdots $
$$
\begin{pmatrix}
A & 0 & 0  & 0 &\dots \\
F_2^*D_P & F_1 & 0 & 0 &\dots  \\
0 & F_2^* & F_1 & 0& \dots \\
0 & 0 & F_2^* & F_1 & \dots\\
\vdots & \vdots & \vdots & \vdots &\ddots
\end{pmatrix}
,
\begin{pmatrix}
B & 0 & 0  & 0 &\dots \\
F_1^*D_P & F_2 & 0 & 0 &\dots  \\
0 & F_1^* & F_2 & 0& \dots \\
0 & 0 & F_1^* & F_2 & \dots\\
\vdots & \vdots & \vdots & \vdots &\ddots
\end{pmatrix}
\text{ and }
\begin{pmatrix}
P & 0 & 0  & 0 &\dots \\
D_P & 0 & 0 & 0 &\dots  \\
0 & I & 0 & 0& \dots \\
0 & 0 & I & 0 & \dots\\
\vdots & \vdots & \vdots & \vdots &\ddots
\end{pmatrix}$$ respectively.
The operators $C_1,C_2\text{ and }C_3:l^2({\mathcal{D}_{P^*}}) \to \tilde{\mathcal{H}}$ are of the form
$$
\begin{pmatrix}
D_{P^*}G_2 & 0 & 0  & 0 &\dots \\
-F_2^*P^* & 0 & 0 & 0 &\dots  \\
0 & 0 & 0 & 0& \dots \\
0 & 0 & 0 & 0 & \dots\\
\vdots & \vdots & \vdots & \vdots &\ddots
\end{pmatrix}
,
\begin{pmatrix}
D_{P^*}G_1 & 0 & 0  & 0 &\dots \\
-F_1^*P^* & 0 & 0 & 0 &\dots  \\
0 & 0 & 0 & 0& \dots \\
0 & 0 & 0 & 0 & \dots\\
\vdots & \vdots & \vdots & \vdots &\ddots
\end{pmatrix}
\text{ and }
\begin{pmatrix}
D_{P^*} & 0 & 0  & 0 &\dots \\
-P^* & 0 & 0 & 0 &\dots  \\
0 & 0 & 0 & 0& \dots \\
0 & 0 & 0 & 0 & \dots\\
\vdots & \vdots & \vdots & \vdots &\ddots
\end{pmatrix}
$$ respectively.Finally the operators $D_1,D_2$ and $D_3$ on $l^2(\mathcal{D}_{P^*})$ are of the form
$$
\begin{pmatrix}
G_1^* & G_2 & 0  & 0 &\dots \\
0 & G_1^* & G_2 & 0 &\dots  \\
0 & 0 & G_1^* & G_2 & \dots \\
0 & 0 &0  & G_1^* & \dots\\
\vdots & \vdots & \vdots & \vdots &\ddots
\end{pmatrix},
\begin{pmatrix}
G_2^* & G_1 & 0  & 0 &\dots \\
0 & G_2^* & G_1 & 0 &\dots  \\
0 & 0 & G_2^* & G_1 & \dots \\
0 & 0 &0  & G_2^* & \dots\\
\vdots & \vdots & \vdots & \vdots &\ddots
\end{pmatrix} \text{ and }
\begin{pmatrix}
0 & I & 0 & 0& \dots \\
0 & 0 & I & 0 & \dots\\
0 & 0 & 0 & I & \dots \\
0 & 0 & 0 & 0 & \dots \\
\vdots & \vdots & \vdots & \vdots &\ddots
\end{pmatrix}
$$ respectively.
\begin{eqnarray*}
&&(V_1C_2+C_1D_2)(a_0,a_1,a_2,\dots)
\\
&=&
V_1(D_{P^*}G_1a_0 \oplus (-F_1^*P^*a_0,0,0,\dots)) +C_1(G_2^*a_0+G_1a_1,G_2^*a_1+G_1a_2,G_2^*a_2+G_1a_3,\dots)
\\
&=&
(AD_{P^*}G_1a_0\oplus((F_2^*D_PD_{P^*}G_1-F_1F_1^*P^*)a_0,-F_2^*F_1^*P^*a_0,0,0,\dots))+
\\
&&
((D_{P^*}G_2G_2^*a_0+D_{P^*}G_2G_1a_1)\oplus(-F_2^*P^*G_2^*a_0-F_2^*P^*G_1a_1,0,0,\dots))
\\
&=&
((AD_{P^*}G_1+D_{P^*}G_2G_2^*)a_0+D_{P^*}G_2G_1a_1)\oplus
\\
&&
((F_2^*D_PD_{P^*}G_1-F_1F_1^*P^*-F_2^*P^*G_2^*)a_0-F_2^*P^*G_1a_1, -F_2^*F_1^*P^*a_0,0,0,\dots)
\end{eqnarray*}
and
\begin{eqnarray*}
&&(V_2C_1+C_2D_1)(a_0,a_1,a_2,\dots)
\\
&=&
V_2(D_{P^*}G_2a_0 \oplus (-F_2^*P^*a_0,0,0,\dots)) + C_2(G_1^*a_0+G_2a_1,G_1^*a_1+G_2a_2,G_1^*a_2+G_2a_3,\dots)
\\
&=&
(BD_{P^*}G_2a_0 \oplus ((F_1^*D_PD_{P^*}G_2-F_2F_2^*P^*)a_0,-F_1^*F_2^*P^*a_0,0,0,\dots)) +
\\
&& ((D_{P^*}G_1G_1^*a_0+D_{P^*}G_1G_2a_1)\oplus(-F_1^*P^*G_1^*a_0-F_1^*P^*G_2a_1,0,0,\dots))
\\
&=&
((BD_{P^*}G_2+D_{P^*}G_1G_1^*)a_0+D_{P^*}G_1G_2a_1) \oplus
\\
&&
((F_1^*D_PD_{P^*}G_2-F_2F_2^*P^*-F_1^*P^*G_1^*)a_0-F_1^*P^*G_2a_1,-F_1^*F_2^*P^*a_0,0,0,\dots).
\end{eqnarray*}
Hence $V_1C_2+C_1D_2=V_2C_1+C_2D_1$ holds if and only if the following holds:
\begin{enumerate}
\item[(a)] $AD_{P^*}G_1+D_{P^*}G_2G_2^*=BD_{P^*}G_2+D_{P^*}G_1G_1^*$, $D_{P^*}G_2G_1=D_{P^*}G_1G_2$,
\item[(b)] $F_2^*D_PD_{P^*}G_1-F_1F_1^*P^*-F_2^*P^*G_2^*=F_1^*D_PD_{P^*}G_2-F_2F_2^*P^*-F_1^*P^*G_1^*$, $F_2^*P^*G_1=F_1^*P^*G_2$, and
\item[(c)] $F_2^*F_1^*P^*=F_1^*F_2^*P^*$.
\end{enumerate}

Using Lemma \ref{tetra} we get that (a) holds if and only if
$$(G_1D_{P^*}+G_2^*D_{P^*}P^*)^*G_1+D_{P^*}G_2G_2^*=(G_2D_{P^*}+G_1^*D_{P^*}P^*)^*G_2+D_{P^*}G_1G_1^* $$ holds. After rearranging this equation we get
$$D_{P^*}G_1^*G_1+PD_{P^*}G_2G_1+D_{P^*}G_2G_2^*=D_{P^*}G_2^*G_2+PD_{P^*}G_1G_2+D_{P^*}G_1G_1^*$$ which is equivalent to saying that
 $D_{P^*}[G_1,G_1^*] +PD_{P^*}G_1G_2=D_{P^*}[G_2,G_2^*]+PD_{P^*}G_2G_1$, which is true since $G_1,G_2$ satisfy equation (\ref{Maa14}) in place of $F_1$ and $F_2$ respectively.
Hence $(a)$ holds.

Note that the first part of equation (b) is equivalent to
\begin{eqnarray*}
&&F_2^*D_PD_{P^*}G_1+F_2F_2^*P^*-F_2^*P^*G_2^*=F_1^*D_PD_{P^*}G_2+F_1F_1^*P^*-F_1^*P^*G_1^*
\\
&\Leftrightarrow& F_2^*(D_PD_{P^*}G_1-P^*G_2^*)+F_2F_2^*P^*=F_1^*(D_PD_{P^*}G_2-P^*G_1^*)+F_1F_1^*P^*
\\
&\Leftrightarrow& F_2^*(F_1^*D_PD_{P^*}-F_2P^*)+F_2F_2^*P^*=F_1^*(F_2^*D_PD_{P^*}-F_1P^*)+F_1F_1^*P^* \; [\text{ by Lemma \ref{tetralem3}}]
\\
&\Leftrightarrow& F_2^*F_1^*D_PD_{P^*}-F_2^*F_2P^*+F_2F_2^*P^*=F_1^*F_2^*D_PD_{P^*}-F_1^*F_1P^*+F_1F_1^*P^*
\\
&\Leftrightarrow& [F_2^*,F_1^*]D_PD_{P^*}+[F_2,F_2^*]P^*=[F_1,F_1^*]P^*,
\end{eqnarray*}
which is true by (\ref{Maa14}).
The second part of $(b)$ follows from Lemma \ref{tetralem4} and from (\ref{Maa14}).
Hence $(b)$ holds.

Equation $(c)$ is simply a consequence of (\ref{Maa14}).
Hence $V_1C_2+C_1D_2=V_2C_1+C_2D_1$. Consequently $R_1$ and $R_2$ commute.

Now let us show that $R_1U=UR_1$. From easy matrix calculations we get that $R_1$ and $U$ commute if and only if
the following holds:
$$V_1V_3=V_3V_1,D_1D_3=D_3D_1 \text{ and } V_1C_3+C_1D_3=V_3C_1+C_3D_1.$$
The first condition holds because $(V_1,V_2,V_3)$ is a tetrablock contraction. The operators $D_1$ and $D_3$ commute because their copy $M_{G_1+G_2^*z}$ and $M_z$ on $H^2_{\mathcal{D}_{P^*}}(\mathbb{D})$ do.
%\begin{eqnarray*}
%D_1D_3(a_0,a_1,a_2,\dots)
%&=&
%D_1(a_1,a_2,a_3,\dots)
%\\
%&=&
%(G_2^*a_1+G_1a_2,G_2^*a_2+G_1a_3,G_2^*a_3+G_1a_4,\dots)
%\\
%&=&
%D_3(G_2^*a_0+G_1a_1,G_2^*a_1+G_1a_2,G_2^*a_2+G_1a_3,G_2^*a_3+G_1a_4,\dots)
%\\
%&=&
%D_3D_1(a_0,a_1,a_2,\dots)
%\end{eqnarray*}
Now
\begin{eqnarray*}
&&V_1C_3+C_1D_3(a_0,a_1,a_2,\dots)
\\
&=&
V_1(D_{P^*}a_0 \oplus (-P^*a_0,0,0,\dots)) + C_1(a_1,a_2,a_3,\dots)
\\
&=&
(AD_{P^*}a_0 \oplus ((F_2^*D_PD_{P^*}-F_1P^*)a_0,-F_2^*P^*a_0,0,0,\dots)) \\
&& + \; (D_{P^*}G_2a_1 \oplus (-F_2^*P^*a_1,0,0,\dots))
\\
&=&
(AD_{P^*}a_0+D_{P^*}G_2a_1 )\oplus ((F_2^*D_PD_{P^*}-F_1P^*)a_0-F_2^*P^*a_1,-F_2^*P^*a_0,0,0,\dots)
\end{eqnarray*}
and
\begin{eqnarray*}
&&(V_3C_1+C_3D_1)(a_0,a_1,a_2,\dots)
\\
&=&
V_3(D_{P^*}G_2a_0 \oplus (-F_2^*P^*a_0,0,0,\dots))+C_3(G_1^*a_0+G_2a_1,G_1^*a_1+G_2a_2,G_1^*a_2+G_2a_3,\dots)
\\
&=&
(PD_{P^*}G_2a_0 \oplus (D_PD_{P^*}G_2a_0,-F_2^*P^*a_0,0,0,\dots)) \\
&& + \; ((D_{P^*}G_1^*a_0+D_{P^*}G_2a_1) \oplus (-P^*G_1^*a_0-P^*G_2a_1,0,0,\dots))
\\
&=&
((PD_{P^*}G_2+D_{P^*}G_1^*)a_0+D_{P^*}G_2a_1) \\
&& \oplus \; ((D_PD_{P^*}G_2-P^*G_1^*)a_0-P^*G_2a_1,-F_2^*P^*a_0,0,0,\dots).
\end{eqnarray*}

Note that the equality of $V_1C_3+C_1D_3$ and $V_3C_1+C_3D_1$ follows trivially from Lemma \ref{tetra}, Lemma \ref{tetralem3} and Lemma \ref{tetralem4}.
The commutativity of  $R_2$ and $U$ can be proved similarly, we skip this.
This completes the proof of $V_1C_2+C_1D_2=V_2C_1+C_2D_1$.

\end{document}

%% file: mypreamble.tex
\theoremstyle{theorem}
    \newtheorem{theorem}{Theorem}
    \newtheorem{lemma}[theorem]{Lemma}

\theoremstyle{definition} % For roman text in the body
    \newtheorem{definition}[theorem]{Definition}

    \newtheorem{remark}[theorem]{Remark}
    \newtheorem{example}[theorem]{Example}
    \newtheorem{exercise}[theorem]{Exercise}

\def\<{\langle}
\def\>{\rangle}

\def\bar{\overline}

\newcommand\mnote[1]{} %off
\newcommand\be{\begin{equation*}}

\newcommand\ee{\end{equation*}}

\newcommand\ben{\begin{equation}}
\newcommand\een{\end{equation}}
\newcommand\bes{\begin{eqnarray*}}
\newcommand\ees{\end{eqnarray*}}

\newcommand\bex{\begin{exercise}}
\newcommand\eex{\end{exercise}}
\newcommand\beg{\begin{example}}
\newcommand\eeg{\end{example}}
\newcommand\benu{\begin{enumerate}}
\newcommand\eenu{\end{enumerate}}
\newcommand\beit{\begin{itemize}}
\newcommand\eeit{\end{itemize}}
\newcommand\berk{\begin{remark}}
\newcommand\eerk{\end{remark}}
\newcommand\bdefn{\begin{defintion}}
\newcommand\edefn{\end{definition}}
\newcommand\bthm{\begin{theorem}}
\newcommand\ethm{\end{theorem}}
\newcommand\bprf{\begin{proof}}
\newcommand\eprf{\end{proof}}
\newcommand\blem{\begin{lemma}}
\newcommand\elem{\end{lemma}}

\newcommand{\sm}{{\raise0.3ex\hbox{$\scriptstyle \setminus$}}}

\def\CHI{\mathchoice%
{\raise2pt\hbox{$\chi$}}%
{\raise2pt\hbox{$\chi$}}%
{\raise1.3pt\hbox{$\scriptstyle\chi$}}%
{\raise0.8pt\hbox{$\scriptscriptstyle\chi$}}}
\def\smalloplus{\raise1pt\hbox{$\,\scriptstyle \oplus\;$}}